\documentclass[10pt]{article}

\font\smallit=cmti10

\usepackage{amssymb,latexsym,amsmath,epsfig,amsthm}
\usepackage{amsfonts}
\usepackage[justification=centering]{caption}
\usepackage{float}
\usepackage{afterpage}
\usepackage{enumerate}
\usepackage{hyperref}

\makeatletter

\renewcommand\section{\@startsection {section}{1}{\z@}
{-30pt \@plus -1ex \@minus -.2ex}
{2.3ex \@plus.2ex}
{\normalfont\normalsize\bfseries\boldmath}}

\renewcommand\subsection{\@startsection{subsection}{2}{\z@}
{-3.25ex\@plus -1ex \@minus -.2ex}
{1.5ex \@plus .2ex}
{\normalfont\normalsize\bfseries\boldmath}}

\renewcommand{\@seccntformat}[1]{\csname the#1\endcsname. }

\makeatother

\newtheorem{theorem}{Theorem}
\newtheorem{lemma}{Lemma}

\newtheorem{corollary}{Corollary}
\newtheorem{definition}{Definition}

\allowdisplaybreaks

\begin{document}

\begin{center}
%% \uppercase{\bf Consecutive primes} \\[4pt]  %% use these two lines rather than the next two
%% \uppercase{\bf which are widely digitally delicate}
{\Large \bf Consecutive primes}
\vskip 5pt
{\Large \bf which are widely digitally delicate}
\vskip 5pt
{\Large \bf and Brier numbers}
\vskip 20pt
{\bf Michael Filaseta}\\
{\smallit Dept.~Mathematics, 
University of South Carolina, 
Columbia, SC 29208, USA}\\
{\tt filaseta@math.sc.edu}\\ 
\vskip 20pt
{\bf Jacob Juillerat}\\
{\smallit Dept.~Mathematics, 
University of North Carolina Pembroke, 
Pembroke, NC 28372, USA}\\
{\tt jacob.juillerat@uncp.edu}\\ 
\vskip 20pt
{\bf Thomas Luckner}\\
{\smallit Dept.~Mathematics, 
University of South Carolina, 
Columbia, SC 29208, USA}\\
{\tt luckner@email.sc.edu}\\ 
\end{center}

\vskip 5pt
\centerline{\phantom{\smallit Received: , Revised: , Accepted: , Published: }} % We will fill in the dates

\vskip 12pt %% change to \vskip 5pt

\vskip 15pt

\centerline{\bf Abstract}
\vskip 5pt\noindent
Making use of covering systems and a theorem of D.~Shiu,
the first and second authors showed that for every positive integer $k$, there exist $k$ consecutive widely digitally delicate primes. They also noted that for every positive integer $k$, there exist $k$ consecutive primes which are Brier numbers.  We show that for every positive integer $k$,  there exist $k$ consecutive primes that are both widely digitally delicate and Brier numbers. 

\pagestyle{myheadings} 
%% \markright{\smalltt INTEGERS: )\hfill}  %%  (remove %% at beginning of line when submitting to journal)
\thispagestyle{empty} 
\baselineskip=12.875pt 
\vskip 30pt
\section{Introduction}

In 1958, R.~M.~Robinson \cite{rob} made tables of primes that were of the form $k\cdot 2^n+1$ for odd integers $1\leq k <100$ and $0\leq n \leq 512$. 
The table found primes for all of these $k$ except 47. 
These tables led to a result of W.~Sierpi\'{n}ski \cite{sier} where he proved there are infinitely many odd positive integers $k$ such that $k\cdot 2^n+1$ is composite for every positive integer $n$.  Such odd $k$ are called Sierpi\'{n}ski numbers.
Sierpi\'{n}ski's proof shows that in fact, for every nonnegative integer $m$, the number 
\[
k = 15511380746462593381 + 36893488147419103230\,m
\]
is a Sierpi\'{n}ski number.  
Here, we have added the condition that $k$ is odd in the definition of a Sierpi\'{n}ski number, as is common, though the original paper of Sierpi\'{n}ski did not have such a condition.
This added condition arose in part due to the desire to find the smallest Sierpi\'{n}ski number as we explain below.

Since Sierpi\'{n}ski's work, there have been a number of attempts to find small explicit examples of Sierpi\'{n}ski numbers. 
In 1962, J. Selfridge (see \cite{slosier}) found the smallest known Sierpi\'{n}ski number, $k=78557$. 
Observe that if we were to remove the condition that $k$ is odd from the definition of a Sierpi\'{n}ski number, then it is likely the number $65536$ would be a smaller Sierpi\'{n}ski number.  The point here is that, for a positive integer $n$, the number $65536 \cdot 2^{n} + 1 = 2^{n+ 16} + 1$ is a prime only if $n + 16$ is a power of $2$.  Thus, the number $65536$ satisfies $65536 \cdot 2^{n} + 1$ is composite for all positive integers $n$ if and only if 
$2^{16} + 1$ is the largest Fermat prime, which is a widely held belief based on heuristic arguments that seems very difficult to prove.

Following the work of Sierpi\'{n}ski \cite{sier}, 
Selfridge used what is known as a covering system to prove $78557$ is a Sierpi\'{n}ski number; covering systems will be discussed in detail later. 
While $78857$ is the smallest known Sierpi\'{n}ski number, it has not been proven to be the smallest one. 
Eliminating $k$ as a possible Sierpi\'{n}ski number can be done by finding a prime of the form $k \cdot 2^{n} + 1$ for some positive integer $n$,
and an effort has been made to find such primes for all odd positive integers $k < 78557$.
In March 2002, there were 17 values of $k$ left to check that are smaller than $78557$ and possible Sierpi\'{n}ski numbers. 
At that time, L.~Heim and D.~Norris started the Seventeen or Bust project which aimed to prove $78557$ is the smallest Sierpi\'{n}ski number (see \cite{pg}). 
The project resulted in showing 11 of the 17 possible Sierpi\'{n}ski numbers less than $78557$ are in fact not Sierpi\'{n}ski numbers. 
Then PrimeGrid took over and showed one more of the remaining possibilities less than $78557$ is not a Sierpi\'{n}ski number in October 2016 leaving 5 more 
possible Sierpi\'{n}ski numbers less than 78557. 
PrimeGrid has continued to run and, as of February 2021, % has now exceed values larger than $31875742$ for $n$ for the remaining $k$.
the remaining possible Sierpi\'{n}ski numbers less than $78557$ are $21181$, $22699$, $24737$, $55459$, and $67607$. 
If $k$ is one of these five remaining numbers, then PrimeGrid has thus far verified that $k \cdot 2^{n} + 1$ is composite for all $n \le 31875742$.  

%Variations of Sierpi\'{n}ski numbers have also become a topic of interest. Bowen \cite{bow} showed, in a short note, that, given a fixed base $b$, there are infinitely many positive integers $k$ for which $k\cdot b^n+1$ is odd and composite for all positive integers $n$. A. Brunner, C. Caldwell, et. al. \cite{brucald} were able to show Bowen's result for all bases $b$, in other words, there exist a $k$ for any base, $b$, where $k\cdot b^n+1$ is composite for all $n$. In the same paper, they found base $b$ Sierpi\'{n}ski numbers where $k, k^2, k^3, \ldots, k^{2^r-1}$ are all base $b$ Sierpi\'{n}ski under some conditions on $b+1$ and Fermat numbers evaluated at $b$. Around the same time Filaseta, Finch and Kozek \cite{Fil} showed for any positive integer $r$ there is a positive integer $k$ such that $k, k^2,\ldots, k^r$ are Sierpi\'{n}ski (not base $b$). 

Around the same time that Sierpi\'{n}ski was working with his new class of numbers, Hans Riesel \cite{ries} showed there are infinitely many odd integers $k$ such that $k\cdot 2^n-1$ is composite for every positive integer $n$.
Riesel showed that for every nonnegative integer $m$, the odd number $k = 509203+11184810\,m$ has this property.  
The odd positive integers $k$ with this property are called Riesel numbers. 

The number $509203$ is conjectured to be the smallest Riesel number (see \cite{slories}). 
The Reisel Sieve Project, analogous to Seventeen or Bust, started in August 2003 to prove the odd positive integers less than $509203$ are not Riesel.  
The project, along with others like PrimeGrid, have reduced the number of possible Riesel numbers $k < 509203$ down to 48.
As of January 2021, PrimeGrid has taken over and shown that these remaining 48 values of $k$ are such that $k \cdot 2^{n} - 1$ is composite for all 
$n \le 11300000$ (see \cite{slories}).

Obtaining Sierpi\'{n}ski and Riesel numbers is typically done by solving a system of linear congruences, where the solution is given in the form of an arithmetic progression.  To find such congruences, one makes use of a covering system. 

\begin{definition}[Covering System]
For nonnegative integers $a_{i}$ and positive integers $b_{i}$, the set 
\[
\{a_1\hskip-4pt \pmod{b_1}, \quad a_2\hskip-4pt \pmod{b_2} ,\quad \ldots, \quad a_m\hskip-4pt \pmod{b_m}\}
\]
is called a \textit{covering system} (or simply a \textit{covering}) if for every integer $n$, we have $n\equiv a_i \pmod{b_i}$ for at least one $i\in\{ 1, 2, \ldots, m\}$.
\end{definition}

By choosing $a_i$ and $b_i$ as in the definition and choosing $p_i$ prime so that $p_i | (2^{b_i}-1)$ and the $p_i$'s are distinct, solutions $k$ to the $m$ congruences
\[
k\cdot 2^{a_1}+ 1 \equiv 0 \hskip -4pt \pmod {p_1}, \quad \ldots, \quad k\cdot 2^{a_m}+ 1 \equiv 0 \hskip -4pt \pmod {p_m}
\]
are Sierpi\'{n}ski, at least for $k>\max_{1\leq i\leq m}\{p_i\}$. Similarly, a covering system can be used to obtain Riesel numbers.

In this paper, we will use a covering system argument to obtain Brier numbers, that is, odd positive integers $k$ that are both Sierpi\'{n}ski numbers and Riesel numbers.  
The existence of such numbers was first established by Brier (see \cite{jacob}). Brier showed that
\[
 29364695660123543278115025405114452910889
 \]
is both a Sierpi\'{n}ski and Riesel number. 

As was done to obtain Sierpi\'{n}ski and Riesel numbers, Brier found covering systems to determine an arithmetic progression of odd positive integers $k$ for which $k \cdot 2^{n} + 1$ and $k \cdot 2^{n} - 1$ are both composite for every positive integer $n$ [1998 unpublished]. 
In August 2009, A.~Wesolowski (see \cite{slobrier}) showed that the smallest Brier number must be larger than $10^9$.
The smallest known Brier number is $3316923598096294713661$ found by C.~Clavier (see \cite{slobrier}) in December 2013. 
Clavier showed that any number in the arithmetic progression
\[
3316923598096294713661 + 3770214739596601257962594704110\,m,
\]
where $m$ is a nonnegative integer, is a Brier number.

One result, which this paper will model its main result after, is the following found in \cite{jacob}.

\begin{theorem}\label{thm1filjui}
For every positive integer $k$, there exist $k$ consecutive primes all of which are Brier numbers.
\end{theorem}

The theorem is a consequence of the arithmetic progression above. 
Since
\[
3316923598096294713661  \quad \text{ and } \quad 3770214739596601257962594704110
 \]
are relatively prime, a result of D.~Shiu \cite{shiu} implies Theorem~\ref{thm1filjui}. 
An extension of Shiu's theorem due to J.~Maynard \cite{maynard} will play an equivalent role in our main result in this paper. 
Shiu showed that in any arithmetic progression containing infinitely many primes, 
there are arbitrarily long sequences of consecutive primes \cite{shiu}. 
For an arithmetic progression $Am+B$, where $A$, $B$ and $m$ are nonnegative integers, to contain infinitely many primes, we want $\gcd(A,B)=1$ and $A > 0$. 
We see then that Theorem~\ref{thm1filjui} follows from Clavier's arithmetic progression of Brier numbers and Shiu's theorem. 
The work of Shiu has been strengthened in the work of W.~D.~Banks, T.~Freiberg and C.~L.~Turnage-Butterbaugh \cite{bft} and J.~Maynard \cite{maynard} (also, see T.~Freiberg \cite{Frei}).
For example, Maynard extends Shiu's work in Theorem 3.3 of \cite{maynard} by showing for every positive integer $k$, in any arithmetic progression $Am+B$, where $A > 0$, $B \ge 0$ and $m \ge 0$ are integers with $A$ and $B$ fixed and $\gcd(A,B) = 1$, a positive proportion of positive integers $\ell$ are such that $p_{\ell}, p_{\ell+1}, \ldots, p_{\ell+k-1}$
are all in the arithmetic progression $Am+B$, where $p_{j}$ denotes the $j^{\rm th}$ prime.

The main result of this paper not only refers to Brier numbers but also refers to widely digitally delicate primes. 
First, we define a digitally delicate prime number.

\begin{definition}[Digitally Delicate]
Let $b$ be an integer $\ge 2$.  
A prime number is called \textit{digitally delicate in base $b$} (or simply \textit{digitally delicate} for $b = 10$)
if changing any base $b$ digit of the prime to another base $b$ digit  results in a composite number.
\end{definition}

In 1979, P.~Erd{\H o}s \cite{erd79} proved that there are infinitely many digitally delicate prime numbers.  
The first of these is 294001.  Hence, 
\[
d94001, \hspace{2em} 2d4001, \hspace{2em} 29d001, \hspace{2em} 294d01, \hspace{2em} 2940d1, \hspace{1em}\text{ and} \hspace{1em} 29400d
\]
are composite or equal to $294001$ for every $d\in \{0,1,2,3,4,5,6,7,8,9\}$. 
The proof of Erd{\H o}s was done using a partial covering system and a sieve argument. 
In 2011, Tao \cite{tao} was able to refine the sieve part of the argument to show that a positive lower asymptotic density of primes are digitally delicate.  
Konyagin \cite{kon} proved a related result for composite numbers using a similar argument showing that, in particular, 
a positive lower asymptotic density of composite numbers which are coprime to $10$ remain composite if any base $10$ digit of the number is changed to another base $10$ digit.  These authors also establish analogous results for arbitrary bases.

In 2021, the first author and J.~Southwick \cite{jerm} considered a more restrictive class of digitally delicate primes in base $10$ and showed a similar result to Tao's for this more restrictive class.  These numbers also are required to be composite when changing any one of the infinitely many leading zeroes in base $10$. 

\begin{definition}[Widely Digitally Delicate]
Let $b$ be an integer $\ge 2$.  
A prime number is called \textit{widely digitally delicate in base $b$} (or simply \textit{widely digitally delicate} for $b = 10$) if changing any base $b$ digit of the prime, including any of the infinitely many leading $0$'s, to another base $b$ digit results in a composite number.
\end{definition}

In the same paper \cite{jerm}, the first author and J.~Southwick elaborate that the result is possible in some other bases, not just base  $10$. 
To see that the collection of widely digitally delicate primes is more refined than digitally delicate primes, consider the digitally delicate prime listed previously, $294001$. 
The prime is not widely digitally delicate since $10294001$ is prime. 
Furthermore, there are no widely digitally delicate primes in the first $10^{9}$ positive integers \cite{jerm}. 
Their result was mentioned in Quanta Magazine \cite{quant} along with the next result by the first and third author \cite{jacob} (also, see \cite{mparker}).
We note that an explicit example of a widely digitally delicate prime has been given by J.~Grantham \cite{Gran} (also, see the comments in \cite{quant}).  The example has $4030$ digits. 

\begin{theorem}\label{jacobthm}
For every positive integer $k$, there exist $k$ consecutive widely digitally delicate primes.
\end{theorem}

In this paper, we use the methods from \cite{jacob} to obtain a result similar to Theorem~\ref{jacobthm} where the primes are also Brier numbers.

\begin{theorem}[Main Theorem] \label{themainthm}
For every positive integer $k$, there exist $k$ consecutive primes 
$p_{\ell}, p_{\ell+1},\ldots,p_{\ell+k-1}$, 
each of which is both a widely digitally delicate prime and a Brier number.
Furthermore, the first primes $p_{\ell}$
in consecutive lists of $k$ such primes
have positive density (depending on $k$) in the set of prime numbers.  In particular, a positive proportion of the primes are both a widely digitally delicate prime and a Brier number.
\end{theorem}

As a corollary to the proof of Theorem~\ref{themainthm}, we will also establish the following.

\begin{corollary}\label{firstcor}
For every positive integer $k$, there exist $k$ consecutive primes that are widely digitally delicate in both base $2$ and base $10$.
\end{corollary}

%%%%%%%%%%%%%%%%%
\section{Some preliminaries and a proof of Corollary~\ref{firstcor}}
In the introduction, we discussed the proof of Theorem~\ref{thm1filjui} based on the use of Shiu's theorem.  
This included finding an arithmetic progression $Am+B$ where $\gcd(A,B)=1$ and every number in the arithmetic progression is a Brier number.  
The goal for  establishing Theorem~\ref{themainthm} is to construct an arithmetic progression $Am+B$, 
with fixed positive relatively prime integers $A$ and $B$, with $A$ divisible by $10$ and having a prime divisor $> 10$, and with a variable nonnegative integer $m$, 
such that every integer $k$ of the form $Am+B$ satisfies 
\begin{itemize}
\item
for every nonnegative integer $n$ and $d \in \{ -9, \ldots, -1, 1, \ldots, 9 \}$, the number $k + d\cdot 10^{n}$ is divisible by at least one prime $p$ dividing $A$, 
\item
for every nonnegative integer $n$, the number $k\cdot 2^n + 1$ is divisible by at least one prime $p$ dividing $A$, and 
\item 
for every nonnegative integer $n$, the number $k\cdot 2^n - 1$ is divisible by at least one prime $p$ dividing $A$.
\end{itemize}
Provided we can also find such a progression that avoids the expressions $k + d\cdot 10^{n}$, $k\cdot 2^n + 1$ and $k\cdot 2^n - 1$ above from being equal to a prime divisor of $A$, then Theorem~\ref{themainthm} will follow directly from 
Maynard's extension of Shiu's theorem mentioned earlier (i.e., Theorem 3.3 of \cite{maynard}).  
Observe that for $m \ge 2$, we have $k\cdot 2^n + 1$ and $k\cdot 2^n - 1$ are both at least $(2A+B) 2^{n} - 1 \ge 2A$, which is larger than any prime divisor of $A$.  
By replacing $B$ then by $2A+B$, we obtain a new arithmetic progression with every element contained in the previous progression, so that the three bulleted items above are still satisfied.  Thus, with $B$ so changed, the expressions $k\cdot 2^n + 1$ and $k\cdot 2^n - 1$ cannot equal a prime divisor of $A$. 
The next lemma allows us to adjust $A$ and $B$ further so that the expressions $k + d\cdot 10^{n}$ in the first bullet above cannot equal a prime divisor of $A$. 

\begin{lemma}\label{initiallemma}
Let $b$ be an integer $\ge 2$. 
Let $A$ and $B$ be positive relatively prime integers such that $A$ has a prime divisor $> b$ and every prime dividing $b$ divides $A$.
Suppose further that for every nonnegative integers $m$ and $k$ and for
\[
d\in \{ -(b-1), -(b-2), \ldots, -1 \} \cup \{ 1, 2, \ldots, b-1 \}, 
\]
there is a prime $p$ dividing $A$ which also divides $Am+B+d\cdot b^k$. 
Then, a subprogression $A_0 m + B_0$ can be found, with $A_0$ and $B_0$ relatively prime and $A$ dividing $A_{0}$,
such that for every nonnegative integers $m$ and $k$ and for $d$ as above, there is a prime $p$ dividing $A_{0}$ with
\[
p \mid 
\big( A_0m+B_0+d\cdot b^k \big)
\quad \text{ and } \quad 
A_0m+B_0+d\cdot b^k\neq \pm p.
\]
\end{lemma}

\begin{proof}
Let $A'$ be the largest positive integer dividing $A$ that is relatively prime to $b$.
Then there is a positive integer $u$ such that $A$ divides $b^{u} A'$.  
Furthermore, taking $v = \phi(A')$ and a positive integer $\ell \ge u$, we see that $A$ divides $b^{\ell (v+1)} - b^{\ell}$.
We take $\ell \ge \max\{ u, 2\}$ sufficiently large so that also $b^{\ell} > A + B$, and set
\[
A_{0} = b^{\ell (v+2)} A 
\qquad \text{and} \qquad
B_{0} = b^{\ell (v+1)} - b^{\ell} + B.
\]
Note that $\ell \ge 2$ ensures that $\ell (v+1) - 1 \ge 2\ell - 1 \ge \ell + 1$, which we use momentarily.
Since $A$ divides $b^{\ell (v+1)} - b^{\ell}$, we see that for every nonnegative integer $m$, there is a nonnegative integer $m'$ such that $A_{0} m + B_{0} = A m' + B$, 
so the arithmetic progression $A_{0} m + B_{0}$ is a subprogression of the arithmetic progression $Am + B$.
By the properties of the progression $Am+B$, we deduce that for every nonnegative integers $m$ and $k$ and for $d\in \{ -(b-1), -(b-2), \ldots, -1 \} \cup \{ 1, 2, \ldots, b-1 \}$, there is a prime $p$ dividing $A$, and hence $A_{0}$, with
$p \mid \big( A_0m+B_0+d\cdot b^k \big)$.
Furthermore, if $k \le \ell (v+1) - 1$, we see that
\begin{align*}
A_0m+B_0+d \cdot b^k 
&= b^{\ell (v+2)} A m + b^{\ell (v+1)} - b^{\ell} + B + d \cdot b^k \\
&\ge b^{\ell (v+1)} - b^{\ell} + B - (b-1) b^{\ell (v+1) - 1} \\
&= b^{\ell (v+1) - 1} - b^{\ell} + B 
\ge b^{\ell + 1} - b^{\ell} + B \\
&> b^{\ell} > A
\end{align*}
for every nonnegative integer $m$.  Also, if $k \ge \ell (v+1)$, then $k \ge 2\ell$ so that
\[
A_0m+B_0+d \cdot b^k \equiv - b^{\ell} + B \pmod{b^{2\ell}}.
\]
Since $b^{\ell} > A+B$, we see that 
\[
- b^{2\ell} + A < b^{\ell} - b^{2\ell} < - b^{\ell} < - b^{\ell} + B < -A.
\]
In this case, $A_0m+B_0+d \cdot b^k$ is in a residue class modulo $b^{2\ell}$ represented by an integer in $(A, b^{2\ell}-A)$.  
In both cases, that is $k \le \ell (v+1) - 1$ and $k \ge \ell (v+1)$, we deduce that since $p \le A$, we have
$A_0m+B_0+d\cdot b^k \neq \pm p$.
\end{proof}

As noted before Lemma~\ref{initiallemma}, we will find an arithmetic progression $Am+B$, with $10\mid A$ and $A$ divisible by a prime $> 10$, satisfying the bullets above.  The comments before Lemma~\ref{initiallemma} together with Lemma~\ref{initiallemma} with $b = 10$ allow us to find a subprogression of $Am+B$ such that every prime $k$ in the subprogression is both widely digitally delicate and a Brier number.  Maynard's extension of Shiu's theorem will then imply Theorem~\ref{themainthm}.

We are now ready to prove Corollary~\ref{firstcor}.

\begin{proof}[Proof of Corollary~\ref{firstcor} (assuming the existence of $Am+B$ as above)]
To establish a prime $k$ is widely digitally delicate in base $2$, 
it suffices to show $k \pm 2^{n}$ is composite for all nonnegative integers $n$. 
Let $k = Am+B$, with $A$ and $B$ as above and $m$ a nonnegative integer.
Fix a nonnegative integer $n$.
Let 
\[
T = n \prod_{p \mid A} (p-1).
\]
Then by the second bullet above, there is a prime $p$ dividing $A$ such that
\[
(Am+B)\cdot2^{T-n} + 1 \equiv 0 \pmod{p}.
 \]
 Note that this congruence implies $p$ is odd. 
 Since $2^T \equiv 1 \pmod p$, we deduce
 \[
 (Am+B)\cdot 2^{-n} + 1 \equiv 0 \pmod{p}.
 \]
 Multiplying both sides of the congruence by $2^n$ gives 
 \[
k + 2^n \equiv 0 \pmod{p}.
 \]
 Thus, for each nonnegative integer $n$, the number $k + 2^{n}$ is divisible by a prime dividing $A$. 
 Similarly, for each nonnegative integer $n$, we can see that the third bullet above implies that the number $k - 2^{n}$ is divisible by a prime dividing $A$. 
 By applying Lemma~\ref{initiallemma} first with $b = 10$ and then with $b = 2$, we see that there is a subprogression of $Am+B$ containing infinitely many primes such that every prime in the subprogression is widely digitally delicate in both base $10$ and base $2$.  
Shiu's theorem applies as before to complete the proof.
\end{proof}

%%%%%%%%%%%%%%%%%
\section{The coverings}
As explained in the previous section, we are interested in finding an arithmetic progression $Am+B$, with $\gcd(A,B) = 1$ and with $A$ divisible by $10$ and some prime $> 10$, satisfying the bulleted items at the beginning of the previous section.  
By replacing $A$ with a power of $A$, we may suppose that $B < A$ and do so.
We will refer to properties (i), (ii) and (iii) analogous to the bullets of the previous section as follows.

\begin{enumerate}[(i)]
\item If $d \in \{ -9, -8, \ldots, -1 \} \cup \{ 1, 2, \ldots, 9 \}$, then each number
in the set 
\[
\mathcal A_{d} = \mathcal A_{d}(A,B)= \big\{ Am+B + d\cdot 10^{n}: m \in \mathbb Z^{+}\cup \{0\},\  n \in \mathbb Z^{+} \cup \{ 0 \} \big\}
\]
is composite.

\item Each number in the set
\[
\mathcal B_{S}= \mathcal B_{S}(A,B)=\lbrace (Am+B) \cdot 2^n +1: m \in \mathbb Z^{+}\cup \{0\}, \ n\in \mathbb{Z}^{+}\cup \{0\} \rbrace\]
is composite.

\item Each number in the set
\[
\mathcal B_{R}=\mathcal B_{R}(A,B)=\lbrace (Am+B) \cdot 2^n -1: m \in \mathbb Z^{+}\cup \{0\},  \ n\in \mathbb{Z}^{+}\cup \{0\} \rbrace
\]
is composite.
\end{enumerate}
In the above, a negative integer is composite if its absolute value is composite.

The statement (i) is exactly the condition we want for each prime in the progression $Am+B$ to be widely digitally delicate. 
Similarly, (ii) and (iii), imply that each prime in the progression $Am+B$ is a Sierpi\'{n}ski number and Riesel number, respectively. 

Note that we want relatively prime $A$ and $B$ satisfying (i), (ii), and (iii). However, we go about this indirectly by finding relatively prime $A_1$ and $B_1$ so that each number in $\mathcal{A}_d(A_1,B_1)$ is composite, by finding relatively prime $A_2$ and $B_2$ so that each number in $\mathcal{B}_S(A_2,B_2)$ is composite, and by finding relatively prime $A_3$ and $B_3$ so that each number in $\mathcal{B}_R(A_3,B_3)$ is composite. Thus, for example, every positive integer that is $B_1$ modulo $A_1$ is such that if we add $d\cdot 10^n$ to the integer, where $n\in \mathbb{Z^+}\cup \{0\}$ and $d \in \{ -9, -8, \ldots, -1 \} \cup \{ 1, 2, \ldots, 9 \}$,  the resulting number is composite. 

Let 
\[
\mathcal{P}(A)=\{ p : p \text{ is prime and } p|A\}.
\]
We will take each of $A_1$, $A_2$, and $A_3$ to be a product of distinct primes. In other words, each $A_j$ is squarefree. The $A_j$'s and $B_j$'s will have the properties
\begin{equation}\label{primes}
\begin{gathered}
\mathcal{P}(A_1)\cap \mathcal{P}(A_2)=\emptyset, \quad \mathcal{P}(A_1)\cap \mathcal{P}(A_3)=\emptyset, \quad \mathcal{P}(A_2)\cap \mathcal{P}(A_3)=\{2,5\},\\
B_2\equiv B_3\hskip-4pt \pmod{10}.
\end{gathered}
\end{equation}
By applying the Chinese Remainder Theorem, we can then establish (i), (ii), and (iii) for some relatively prime $A$ and $B$ by taking $A=A_1A_2A_3/10$ and $B\in [0,A-1]$ so that
\begin{equation} \label{cong}
\begin{gathered}
B\equiv B_1 \hskip-4pt \pmod {A_1}, \quad B\equiv B_2 \hskip-4pt \pmod {A_2/10}, \quad B\equiv B_3 \hskip-4pt \pmod {A_3/10}, \\
B\equiv B_2\equiv B_3 \hskip-4pt \pmod {10}. 
\end{gathered}
\end{equation}

The first and third authors \cite{jacob} constructed $A_1$ squarefree and $B_1$ so that (i) holds with $\mathcal{A}_d(A,B)$ replaced by $\mathcal{A}_d(A_1,B_1)$. Thus, this paper will focus on constructing the pair $(A_2,B_2)$ as well as the pair $(A_3,B_3)$ such that \eqref{primes} holds and (ii) and (iii) hold with $\mathcal{B}_S(A,B)$ and $\mathcal{B}_R(A,B)$ replaced by $\mathcal{B}_S(A_2,B_2)$ and $\mathcal{B}_R(A_3,B_3)$, respectively.

Since the construction of the pairs $(A_2,B_2)$ and $(A_3,B_3)$ is similar, we will discuss the construction of the pair $(A_2,B_2)$ and then explain how this translates to constructing the pair $(A_3,B_3)$. To show that $(A_2m+B_2)\cdot 2^n+1$ is composite for all nonnegative integers $n$, we will show that for each nonnegative integer $n$,  there is a prime, $p\in \mathcal{P}(A_2)$ such that $p$ divides $(A_2m+B_2)\cdot 2^n+1$.  We will choose $A_2$ and $B_2$ large enough so that each number of the form $(A_2m+B_2)\cdot 2^n+1$, with $m$ and $n$ in $\mathbb{Z^+}\cup \{0\}$, is greater than each prime in $\mathcal{P}(A_2)$. Thus, every number of the form $(A_2m+B_2)\cdot 2^n+1$ will be composite.

For a prime $p\in \mathcal{P}(A_2)$, observe that $(A_2m+B_2)\cdot 2^n+1$ is divisible by $p$ if and only if $B_2\cdot 2^n+1$ is divisible by $p$.  Initially, we do not know the values of $A_2$ and $B_2$; we want to construct them.
The idea then is to find a finite set $\mathcal P_2$ of primes so that for some positive integer $B_2$ and all nonnegative integers $n$, the number $B_2\cdot 2^n+1$ is divisible by one of the primes in $\mathcal P_2$. Then $A_2$ will be determined by taking $A_2$ to be the product of the primes in $\mathcal P_2$ so that $\mathcal{P}(A_2) = \mathcal P_2$.  We want to construct $B_2$ as above in such a way that $\gcd(A_2,B_2) = 1$. The focus now is on finding the set $\mathcal P_2$ and how to construct $B_2$ from this set.

For an odd prime $p$ and an arbitrary integer $a$, we can determine $B_2 \equiv -2^{-a} \pmod{p}$ so that
\[
B_2 \cdot 2^{n} + 1 \equiv 0 \pmod{p},
\]
when $n \equiv a \pmod{b}$ with $b = \text{ord}_{p}(2)$.  The idea now is to determine primes $p$ (our set $\mathcal P_2$) so that the orders of $2$ modulo these primes and appropriate choices for $a$ as above provide us with a list of congruence classes $n \equiv a \pmod{b}$ that form a covering system for the integers.  In this way, for every nonnegative integer $n$, there will be a prime $p$ such that $p$ divides $B_2 \cdot 2^{n} + 1$, where $p$ depends on a congruence class satisfied by $n$.

We now use an example to illustrate how a congruence class in the covering system gives a congruence class for $B_2$ to satisfy.
We take $p = 5$.  The order of $2$ modulo $5$ is $4$, so the modulus for the congruence on $n$ will be $4$. 
Suppose we want the congruence $n\equiv 0 \pmod 4$ in the covering system for $n$. If we let $B_2 \equiv -2^{-0} \equiv 4 \pmod 5$, then for some integer $t$, we have
\[
B_2\cdot 2^n+1
= B_2\cdot 2^{4t}+1
\equiv 4 \cdot 1 +1 \equiv 0 \hskip-4pt\pmod{5}. 
\]
Thus, whenever $n\equiv 0 \pmod 4$, the number $B_2\cdot 2^n+1$ is divisible by $5$.  Then $n\equiv 0 \pmod 4$ is part of the covering system we want to help establish (ii), where we want 
\[
A_2\equiv 0 \hskip-4pt \pmod 5 \quad \text{ and } \quad B_2\equiv 4 \hskip-4pt\pmod 5.
\]

In order for \eqref{primes} to hold, observe that we need $5 \not\in \mathcal{P}(A_1)$. 
The value of $A_1$ is given in \cite{jacob}, and throughout this paper, we avoid using primes in $\mathcal{P}(A_1)$.  It is the case that $5 \not\in \mathcal{P}(A_1)$, so using $5$ as above is permissible.
We note that the primes $2$ and $5$ were not in $\mathcal{P}(A_1)$ because, in \cite{jacob}, the authors were interested in choosing moduli for the coverings based on the order of $10$ modulo the primes dividing $A_{1}$.  This led them to avoiding the primes $2$ and $5$ for which no such order exists.  Similarly in this paper, we will want $2$ to have an order modulo each of the primes dividing $A_2$ or $A_3$, and thus we will seemingly want to avoid the prime $2$.  However, we have already indicated that we want $A_j m + B_j$ to be odd for $j \in \{ 2, 3 \}$, so we are taking $2$ to be in both $\mathcal{P}(A_2)$ and $\mathcal{P}(A_3)$ with the added conditions that $B_2$ and $B_3$ are $1$ modulo $2$.

Now, suppose we want to show that $(A_2m+B_2) \cdot 2^n + 1$ is composite as in (ii) whenever $n \equiv 2 \pmod{4}$.
Since $5$ is the only prime $p$ with $2$ of order $4$ modulo $p$, the idea is to work modulo $8$ instead and show that $(A_2m+B_2) \cdot 2^n + 1$ is composite when $n \equiv 2 \pmod{8}$ and when $n \equiv 6 \pmod{8}$.  Each of these will require a number of congruence classes, particularly since we want to avoid primes in $\mathcal P(A_{1})$ and the only prime $p$ with $2$ of order $8$ modulo $p$ is $p = 17 \in \mathcal P(A_{1})$.
For the discussion now, we consider the case of showing $(A_2m+B_2) \cdot 2^n + 1$ is composite when $n \equiv 2 \pmod{8}$. 

By the above arguments, we want to find primes so that if $B_2$ satisfies certain congruence classes, then whenever $n \equiv 2 \pmod{8}$, one of these primes divides $B_2 \cdot 2^n + 1$.
The primes we found are given in the second column in Table \ref{covereg}.
Momentarily, we will describe how we determined the primes more clearly, but if 
$n \equiv a \pmod{b}$ and $p$ are the columns in a row of Table \ref{covereg}, then the order of $2$ modulo $p$ is $b$.
To see that every integer $n \equiv 2 \pmod{8}$ satisfies a congruence class in the first column, observe that if $n \equiv 128 \pmod{144}$, then $n$ satisfies one of the last 3 congruence classes in the first column of Table \ref{covereg}.  With the prior two congruence classes modulo $144$, every $n \equiv 42 \pmod{48}$ will satisfy one of the last 5 congruence classes in the first column.  If $n \equiv 74 \pmod{96}$, then it satisfies one of the congruence classes in rows 4-6 in the first column.  Thus, if $n \equiv 26 \pmod{48}$, then it satisfies one of the congruence classes in rows 3-6 in the first column. Now, if $n \equiv 10 \pmod{16}$, then it satisfies one of the congruence classes in rows 2-10 in the first column. Since integers which are $2$ modulo $16$ satisfy the first congruence class in the first column, we see that every integer $n \equiv 2 \pmod{8}$ satisfies at least one congruence class in the first column of Table \ref{covereg}.

 \begin{table}[!hbt]
  \centering
 \caption{Congruence classes used to satisfy $n\equiv 2 \pmod 8$}\label{covereg}
\begin{tabular}{|c|c|}
\hline Congruence class & prime \\
\hline \hline
$n \equiv 2\pmod{16}$ & 257\\
\hline
$n \equiv 10 \pmod{48}$ & 673\\
\hline
$n \equiv 26 \pmod {96}$ & 22253377\\
\hline
$n \equiv 74 \pmod {288}$ & 1153\\
\hline
$n \equiv 170\pmod {288}$ & 6337\\
\hline
$n \equiv 266 \pmod {288}$ & 38941695937\\
\hline
$n \equiv 42 \pmod {144}$ & 577\\
\hline
$n \equiv 90 \pmod {144}$ & 487824887233\\
\hline
$n \equiv 138 \pmod {432}$ & 4261383649\\
\hline
$n \equiv 282\pmod {432}$ & 209924353\\
\hline
$n \equiv 426 \pmod{432}$ & 24929060818265360451708193\\
\hline
\end{tabular}
\end{table}

So we will choose $A_2$ so that it is divisible by each prime in the second column of Table \ref{covereg}.  Corresponding to each congruence $n \equiv a \pmod{b}$ and prime $p$ in a row, we take $B_2 \equiv -2^{-a} \pmod{p}$ as noted earlier.  Then whenever $n \equiv 2 \pmod{8}$ and $m \in \mathbb Z$, we have $(A_2m+B_2) \cdot 2^n + 1$ is divisible by a prime in the second column of Table \ref{covereg}.

Next, we describe more clearly where our choice of primes came from. The idea in the example above is to find a system $\mathcal C$ of congruence classes such that every $n \equiv 2 \pmod{8}$ satisfies a congruence $n \equiv a \pmod{b}$ in $\mathcal C$.  Each such congruence class will correspond to a prime $p \not\in \mathcal P(A_1)$ for which $2$ has order $b$ modulo $p$. 
To find the primes $p$ for which $2$ has order $b$ modulo $p$, where $b$ is a modulus we want to use for $\mathcal C$, we look at the prime divisors of $\Phi_{b}(2)$, where $\Phi_{b}(x)$ is the $b^{\rm th}$ cyclotomic polynomial. As noted in \cite{jacob} and \cite{jj} (with $2$ replaced by $10$), each prime divisor $p$ of $\Phi_{b}(2)$ either will be the largest prime divisor of $b$ or will satisfy that $2$ has order $b$ modulo $p$.  Thus, looking at prime divisors of $\Phi_{b}(2)$, which are not in $\mathcal{P}(A_1)$, determines whether we can use the modulus $b$ for $\mathcal C$ and how many times we can use $b$ as a modulus.
Thus, in general, we found what moduli were possible for our systems of congruence classes by looking at the number of distinct prime divisors of $\Phi_{b}(2)$, not in $\mathcal{P}(A_1)$, for different choices of $b$, and then we determined our covering systems based on this information.

Observe in Table~\ref{covereg} that we could have interchanged the primes with $2$ of a given order $b$.  For example, for the congruences $n \equiv 74 \pmod{288}$, $n \equiv 170 \pmod{288}$, and $n \equiv 266 \pmod{288}$, we could have selected the primes in the right-most column as $38941695937$, $1153$, and $6337$, respectively.

Tables~\ref{orderofprimes} and \ref{orderofunusedprimes1} describe the results of looking at prime factors of $\Phi_{b}(2)$ to determine the moduli $b$ we can use for our covering systems and the number of times we can use each modulus (that is, the number of distinct prime factors of $\Phi_{b}(2)$ that do not divide $b$ and are not in $\mathcal{P}(A_1)$). Table \ref{orderofprimes} indicates the number of distinct prime factors of $\Phi_{b}(2)$ that do not divide $b$ and are not in $\mathcal{P}(A_1)$ but are in $\mathcal{P}(A_2)
\cup \mathcal{P}(A_3)$. Table \ref{orderofunusedprimes1} shows the remaining number of distinct prime factors of $\Phi_{b}(2)$ which we know exist and do not divide $b$ and are not in $\mathcal{P}(A_1)$ for each modulus. If a modulus $b$ appears in Table~\ref{orderofprimes} but not in Table~\ref{orderofunusedprimes1}, then all the distinct prime factors of $\Phi_{b}(2)$ that do not divide $b$ and are not in $\mathcal{P}(A_1)$ have been used in $\mathcal{P}(A_2)
\cup \mathcal{P}(A_3)$.

The above describes how we obtained a covering system for determining $A_2$ and $B_2$ so that (ii) holds.  We obtain a covering system for determining $A_3$ and $B_3$ similarly so that (iii) holds. In this case, each congruence $n \equiv a \pmod{b}$ corresponds to an odd prime $p$ dividing $A_3$, with $b$ equal to the order of $2$ modulo $p$, and we want 
\begin{equation}\label{B3}
B_3 \equiv 2^{-a} \pmod{p}
\end{equation}
so that
\begin{equation*}
B_3 \cdot 2^{n} - 1 \equiv 0 \pmod{p}
\end{equation*}
for $n \equiv a \pmod{b}$.  
%% By \eqref{primes}, since $2$ has order $4$ modulo $5$ and not modulo any other primes not in $\mathcal{P}(A_1)$, the modulus $4$ is the only modulus we can use in our coverings to obtain both (ii) and (iii). 

For $p = 5$ above, we took 
$B_2 \equiv 4 \pmod{5}$.
Also, $B_2$ is odd, so we have $B_2 \equiv 9 \pmod{10}$.   
For \eqref{primes} to hold, we therefore want 
$B_3 \equiv 9 \pmod{10}$.
In particular, $B_3 \equiv 4 \pmod{5}$.  From \eqref{B3}, with $p = 5$, 
we want $a = 2$. As the order of $2$ modulo $5$ is $4$, the congruence we want associated with $p = 5$ and (iii) is $n \equiv 2 \pmod{4}$.

Each prime counted in Table~\ref{orderofprimes} in the last column, besides $p = 5$ with $b = 4$, can be used to provide a congruence class for the covering system corresponding to either (ii) or (iii) and not both, due to the restriction made in \eqref{primes} that $\mathcal{P}(A_2) \cap \mathcal{P}(A_3) = \{ 2,5 \}$. 
The complete covering systems used for determining $A_j$ and $B_j$, for $j \in \{2, 3\}$, can be found in the appendix.  In the next section, we explain how we verified that the congruence classes we tabulated for the covering systems in fact form covering systems.

%%%%%%%%%%%%%%%%%%%%%

\section{Verifying the Covering Systems}
Consider the collection of congruence classes
\[
\mathcal {C}_0=\{ 0\hskip-4pt \pmod 3,\hskip4pt 1\hskip-4pt \pmod 3,\hskip4pt 2\hskip-4pt\pmod 9, \hskip4pt 5 \hskip-4pt\pmod 9, \hskip4pt 8\hskip-4pt\pmod 9\}.
\]
We can verify $\mathcal{C}_0$ is a covering system as follows. Every nonnegative integer is either 0, 1 or 2 $\pmod 3$. The congruence classes $0 \pmod 3$ and $1 \pmod 3$ are in $\mathcal{C}_0$, so the case when an integer is $2 \pmod {3}$ is left to satisfy. Every integer that is $2\pmod 3$ is either $2\pmod 9$, $5 \pmod 9$, or $8 \pmod 9$. These congruence classes modulo 9 are in $\mathcal{C}_0$. Thus, $\mathcal{C}_0$ is a covering system. 

Due to the complexity of the coverings in the previous section, this method for verifying in general a 
\[
\mathcal{C}=\{a_1\hskip-4pt \pmod{b_1}, \quad a_2\hskip-4pt \pmod{b_2} ,\quad \ldots, \quad a_m\hskip-4pt \pmod{b_m}\}
\]
is a covering is quite time consuming. An alternate way of verifying $\mathcal{C}$ is a covering is to check every integer in $[0, \ell-1]$, where $\ell=\text{lcm}(b_i)$. To see this, suppose that every integer in $[0, \ell-1]$ is in at least one congruence class in $\mathcal{C}$. We want to show that all integers $k$ are in at least one congruence class in $\mathcal{C}$.  We can rewrite $k$ as 
\[
k=q\cdot \ell+r
\]
where $q$ and $r$ are integers with $0\leq r\leq \ell-1$. Since  $r\in[0,\ell-1]$, we have that $r$ is in at least one congruence class, $a\pmod{b}$, in $\mathcal{C}$. As a consequence of $b$ dividing $\ell$, we deduce
\[
k\equiv r\equiv a \hskip-4pt\pmod{b}.
\]
Thus, $k$ is in the congruence class $a\pmod b$ in $\mathcal{C}$, as we wanted.

%%% Not necessary for paper but good for dissertation
In the example above with $\mathcal{C}_0$, this process is emulated by checking that every $k\in [0, 8]$ satisfies a congruence class in $\mathcal C_0$. Table \ref{tablenumcong} confirms that $\mathcal C_0$ is a covering by listing a congruence class each $k$ satisfies in $\mathcal{C}_0$ in the second column. 
\begin{center}
\begin{table}[!hbt]
\caption{Verifying the Covering $\mathcal C_0$}\label{tablenumcong}
\centering
\begin{tabular}{|c|c|}
\hline
$k$ & Congruence Classes in $\mathcal C_0$ \\ 
\hline \hline
$0$ & $0  \pmod 3$ \\ \hline 
$1$ & $1 \pmod 3$ \\ \hline 
$2$ & $2 \pmod 9$ \\ \hline 
$3$ & $0\pmod 3$ \\ \hline 
$4$ & $1 \pmod 3$ \\ \hline 
$5$ & $5 \pmod 9$ \\ \hline 
$6$ & $0 \pmod 3$ \\ \hline 
$7$ & $1\pmod 3$ \\ \hline 
$8$ & $8 \pmod 9$ \\ \hline 
\end{tabular} 
\end{table}
\end{center}

While this method works and is easily implemented by a computer, if $\ell=\text{lcm}(b_i)$ is too large, this process takes a substantial amount of time. In the coverings used in this paper for Sierpi\'{n}ski and Riesel numbers, the least common multiples are 236107872000 and 922078080000, respectively. Also, the number of congruence classes in the coverings are 447 and 459, respectively. Thus, an alternative verification method, as seen in the paper  \cite{jacob}, was used. 

Let $\mathcal{C}$ be a set of congruence classes $a_i\pmod{b_i}$. Let $w$ be a positive integer and $u$ be an integer in $[0,w-1]$. 
Let $\mathcal{C}_u$ be the congruence classes $a_i \pmod {b_i}$ in $\mathcal{C}$ such that 
\[
a_i\equiv u\pmod{\gcd{(b_i,w)}}.
\]
If $|\mathcal{C}_u|=0$, then $\mathcal{C}$ is not a covering. Now consider the case that $|\mathcal{C}_u|>0$. Let $\ell'$ be the lcm of the moduli in $\mathcal{C}_u$, and set $d=\gcd(w,\ell')$. In \cite{jacob}, the authors show the following. 
\begin{lemma}\label{lem1}
With the above notation, if for each $u\in[0,w-1]$, every
\[
k=wt+u, \hspace{.5cm} \text{with } t\in[0,(\ell'/d)-1]\cap \mathbb{Z},
\]
satisfies a congruence class in $\mathcal{C}_u$, then $\mathcal{C}$ is a covering system. 
\end{lemma}
Let $\mathcal{C}$ be one of the systems of congruence classes created in the previous section, that is for constructing arithmetic progressions for either Sierpi\'{n}ski or Riesel numbers. We take $w=4\cdot 3\cdot 5\cdot q$ where $q$ is the largest prime dividing the least common multiple of the moduli in $\mathcal{C}$  as done in \cite{jacob}. Applying Lemma~\ref{lem1} with this choice of $w$ allowed us to easily verify our congruence classes form a covering system.

\vskip 10pt
\centerline{\textbf{\large Appendix}}
\vskip 7pt \noindent

To aid in verifying the computations in this paper, all the data in this appendix can be found in \cite{Fildatatwo} as lists suitable for computations. 
This appendix begins with Table~\ref{orderofprimes}, which lists the number of distinct primes used, $L=L(b)$, that divide $\Phi_b(2)$, do not divide $b$, and are not in $\mathcal{P}(A_1)$. Not the * indicates the prime 5 is used once in each covering.
The $b$ listed correspond to moduli used in our coverings.  Table~\ref{orderofunusedprimes1} gives a lower bound on number of distinct primes, $M=M(b)$, that divide $\Phi_b(2)$, do not divide $b$, are not in $\mathcal{P}(A_1)$, and are not used in either covering. The * in this table represents an $M(b)$ that is a lower bound; that is, we did not completely factor $\Phi_b(2)$. 

Tables~\ref{covSier} and \ref{covRies} list the congruence classes $n \equiv a \pmod{b}$ that form the covering systems we obtained for $k\cdot 2^n+ 1$ and $k\cdot 2^n-1$ respectively.  
Recall that, with the order $b$ of $2$ fixed, the choice of a prime $p$ associated with a given congruence class does not matter.  Thus, for example, $\Phi_{64}(2)$ has exactly two prime factors, $641$ and $6700417$, neither of which are in $\mathcal P(A_1)$.  These two primes are in the count for $L(64)$ in Table~\ref{orderofprimes}. We associate some ordering of these two primes.  The second column of Table~\ref{covSier} indicates that one of these primes is associated with the congruence $n \equiv 22 \pmod{64}$ and the other is associated with the congruence $n \equiv 54 \pmod{64}$.  Note that we do not attempt to clarify which of the two primes is associated with which congruence as it does not matter.

\newpage

%%%%% Table of all used primes

\begin{table}[tp]
\footnotesize
\centering
\caption{Number of primes used in both coverings, $L=L(b)$}\label{orderofprimes}
\begin{minipage}{1.9 cm} 
\centering 
% % 
\begin{tabular}{|c|c|} 
\hline
$b$ & $L$ \\  \hline
\hline 
$4$ & $1^*$ \\ 
 \hline 
$13$ & $1$ \\ 
 \hline 
$16$ & $1$ \\ 
 \hline 
$25$ & $2$ \\ 
 \hline 
$26$ & $1$ \\ 
 \hline 
$27$ & $1$ \\ 
 \hline 
$32$ & $1$ \\ 
 \hline 
$33$ & $1$ \\ 
 \hline 
$35$ & $1$ \\ 
 \hline 
$39$ & $1$ \\ 
 \hline 
$40$ & $1$ \\ 
 \hline 
$48$ & $1$ \\ 
 \hline 
$50$ & $1$ \\ 
 \hline 
$55$ & $1$ \\ 
 \hline 
$64$ & $2$ \\ 
 \hline 
$65$ & $1$ \\ 
 \hline 
$66$ & $1$ \\ 
 \hline 
$70$ & $1$ \\ 
 \hline 
$75$ & $2$ \\ 
 \hline 
$77$ & $1$ \\ 
 \hline 
$78$ & $1$ \\ 
 \hline 
$80$ & $1$ \\ 
 \hline 
$81$ & $3$ \\ 
 \hline 
$88$ & $1$ \\ 
 \hline 
$91$ & $2$ \\ 
 \hline 
$96$ & $1$ \\ 
 \hline 
$99$ & $2$ \\ 
 \hline 
$100$ & $2$ \\ 
 \hline 
$104$ & $2$ \\ 
 \hline 
$105$ & $3$ \\ 
 \hline 
$108$ & $1$ \\ 
 \hline 
$110$ & $2$ \\ 
 \hline 
$112$ & $2$ \\ 
 \hline 
$117$ & $4$ \\ 
 \hline 
$120$ & $1$ \\ 
 \hline 
$121$ & $1$ \\ 
 \hline 
$125$ & $2$ \\ 
 \hline 
$130$ & $2$ \\ 
 \hline 
$132$ & $2$ \\ 
 \hline 
$135$ & $2$ \\ 
 \hline 
$140$ & $2$ \\ 
 \hline 
$143$ & $3$ \\ 
 \hline 
$144$ & $2$ \\ 
 \hline 
$150$ & $1$ \\ 
 \hline 
$154$ & $2$ \\ 
 \hline 
$156$ & $1$ \\ 
 \hline 
$160$ & $2$ \\ 
 \hline 
$162$ & $1$ \\ 
 \hline 
$165$ & $1$ \\ 
 \hline 
$169$ & $3$ \\ 
 \hline
 $175$ & $3$ \\ 
 \hline 
\end{tabular} 
\end{minipage} 
\begin{minipage}{1.9 cm} 
\centering 
% % 
\begin{tabular}{|c|c|} 
\hline
$b$ & $L$ \\  \hline
\hline 
$176$ & $3$ \\ 
 \hline 
$182$ & $3$ \\ 
 \hline 
$195$ & $1$ \\ 
 \hline 
$198$ & $2$ \\ 
 \hline 
$200$ & $4$ \\ 
 \hline 
$208$ & $1$ \\ 
 \hline 
$216$ & $2$ \\ 
 \hline 
$220$ & $2$ \\ 
 \hline 
$224$ & $2$ \\ 
 \hline 
$225$ & $3$ \\ 
 \hline 
$231$ & $1$ \\ 
 \hline 
$234$ & $1$ \\ 
 \hline 
$240$ & $2$ \\ 
 \hline 
$242$ & $2$ \\ 
 \hline 
$250$ & $2$ \\ 
 \hline 
$260$ & $3$ \\ 
 \hline 
$264$ & $2$ \\ 
 \hline 
$270$ & $3$ \\ 
 \hline 
$273$ & $3$ \\ 
 \hline 
$275$ & $2$ \\ 
 \hline 
$280$ & $1$ \\ 
 \hline 
$286$ & $3$ \\ 
 \hline 
$288$ & $3$ \\ 
 \hline 
$297$ & $2$ \\ 
 \hline 
$300$ & $3$ \\ 
 \hline 
$308$ & $3$ \\ 
 \hline 
$312$ & $1$ \\ 
 \hline 
$324$ & $4$ \\ 
 \hline 
$325$ & $2$ \\ 
 \hline 
$330$ & $2$ \\ 
 \hline 
$336$ & $3$ \\ 
 \hline 
$338$ & $2$ \\ 
 \hline 
$350$ & $2$ \\ 
 \hline 
$351$ & $5$ \\ 
 \hline 
$352$ & $2$ \\ 
 \hline 
$360$ & $2$ \\ 
 \hline 
$363$ & $4$ \\ 
 \hline 
$364$ & $2$ \\ 
 \hline 
$375$ & $1$ \\ 
 \hline 
$378$ & $4$ \\ 
 \hline 
$385$ & $3$ \\ 
 \hline 
$390$ & $2$ \\ 
 \hline 
$396$ & $4$ \\ 
 \hline 
$416$ & $5$ \\ 
 \hline 
$420$ & $2$ \\ 
 \hline 
$429$ & $3$ \\ 
 \hline 
$432$ & $3$ \\ 
 \hline 
$440$ & $4$ \\ 
 \hline 
$448$ & $2$ \\ 
 \hline
 $450$ & $2$ \\ 
 \hline 
 $455$ & $2$ \\ 
 \hline
\end{tabular} 
\end{minipage} 
\centering
\begin{minipage}{1.9 cm} 
\centering 
% % 
\begin{tabular}{|c|c|} 
\hline
$b$ & $L$ \\  \hline
\hline  
$462$ & $1$ \\ 
 \hline 
$468$ & $5$ \\ 
 \hline 
$480$ & $2$ \\ 
 \hline 
$484$ & $4$ \\ 
 \hline 
$486$ & $3$ \\ 
 \hline 
$495$ & $2$ \\ 
 \hline 
$507$ & $1$ \\ 
 \hline 
$520$ & $5$ \\ 
 \hline 
$528$ & $2$ \\ 
 \hline 
$540$ & $4$ \\ 
 \hline 
$546$ & $2$ \\ 
 \hline 
$550$ & $1$ \\ 
 \hline 
$560$ & $4$ \\ 
 \hline 
$572$ & $6$ \\ 
 \hline 
$585$ & $4$ \\ 
 \hline 
$594$ & $4$ \\ 
 \hline 
$616$ & $3$ \\ 
 \hline 
$624$ & $3$ \\ 
 \hline 
$625$ & $4$ \\ 
 \hline 
$637$ & $7$ \\ 
 \hline 
$648$ & $7$ \\ 
 \hline 
$650$ & $5$ \\ 
 \hline 
$660$ & $5$ \\ 
 \hline 
$672$ & $2$ \\ 
 \hline 
$676$ & $5$ \\ 
 \hline 
$693$ & $3$ \\ 
 \hline 
$700$ & $3$ \\ 
 \hline 
$702$ & $2$ \\ 
 \hline 
$704$ & $4$ \\ 
 \hline 
$715$ & $3$ \\ 
 \hline 
$720$ & $2$ \\ 
 \hline 
$726$ & $3$ \\ 
 \hline 
$728$ & $2$ \\ 
 \hline 
$750$ & $3$ \\ 
 \hline 
$756$ & $2$ \\ 
 \hline 
$770$ & $3$ \\ 
 \hline 
$780$ & $5$ \\ 
 \hline 
$784$ & $5$ \\ 
 \hline 
$792$ & $2$ \\ 
 \hline 
$810$ & $6$ \\ 
 \hline 
$825$ & $4$ \\ 
 \hline 
$832$ & $3$ \\ 
 \hline 
$840$ & $3$ \\ 
 \hline 
$845$ & $3$ \\ 
 \hline 
$858$ & $2$ \\ 
 \hline 
$864$ & $2$ \\ 
 \hline 
$880$ & $4$ \\ 
 \hline 
$896$ & $2$ \\ 
 \hline 
 $910$ & $4$ \\ 
 \hline
 $924$ & $2$ \\ 
 \hline  
 $936$ & $4$ \\ 
 \hline 
\end{tabular} 
\end{minipage} 
\centering
\begin{minipage}{1.9 cm} 
\centering 
% % 
\begin{tabular}{|c|c|} 
\hline
$b$ & $L$ \\  \hline
\hline 
$960$ & $4$ \\ 
 \hline 
$968$ & $4$ \\ 
 \hline 
$975$ & $4$ \\ 
 \hline 
$980$ & $5$ \\ 
 \hline 
$1014$ & $3$ \\ 
 \hline 
$1040$ & $4$ \\ 
 \hline 
$1053$ & $3$ \\ 
 \hline 
$1056$ & $4$ \\ 
 \hline 
$1078$ & $3$ \\ 
 \hline 
$1080$ & $3$ \\ 
 \hline 
$1089$ & $4$ \\ 
 \hline 
$1092$ & $5$ \\ 
 \hline 
$1100$ & $7$ \\ 
 \hline 
$1120$ & $4$ \\ 
 \hline 
$1125$ & $3$ \\ 
 \hline 
$1134$ & $3$ \\ 
 \hline 
$1155$ & $2$ \\ 
 \hline 
$1170$ & $5$ \\ 
 \hline 
$1183$ & $3$ \\ 
 \hline 
$1188$ & $6$ \\ 
 \hline 
$1232$ & $2$ \\ 
 \hline 
$1248$ & $3$ \\ 
 \hline 
$1250$ & $2$ \\ 
 \hline 
$1260$ & $3$ \\ 
 \hline 
$1274$ & $3$ \\ 
 \hline 
$1300$ & $7$ \\ 
 \hline 
$1320$ & $4$ \\ 
 \hline 
$1344$ & $4$ \\ 
 \hline 
$1350$ & $3$ \\ 
 \hline 
$1352$ & $4$ \\ 
 \hline 
$1365$ & $4$ \\ 
 \hline 
$1375$ & $3$ \\ 
 \hline 
$1386$ & $1$ \\ 
 \hline 
$1400$ & $2$ \\ 
 \hline 
$1404$ & $7$ \\ 
 \hline 
$1408$ & $3$ \\ 
 \hline 
$1430$ & $2$ \\ 
 \hline 
$1440$ & $4$ \\ 
 \hline 
$1452$ & $2$ \\ 
 \hline 
$1456$ & $4$ \\ 
 \hline 
$1485$ & $3$ \\ 
 \hline 
$1512$ & $1$ \\ 
 \hline 
$1540$ & $4$ \\ 
 \hline 
$1560$ & $3$ \\ 
 \hline 
$1568$ & $2$ \\ 
 \hline 
$1584$ & $2$ \\ 
 \hline 
$1620$ & $2$ \\ 
 \hline
 $1625$ & $5$ \\ 
 \hline 
 $1638$ & $3$ \\ 
 \hline 
 $1664$ & $2$ \\ 
 \hline 
$1680$ & $4$ \\ 
 \hline  
\end{tabular} 
\end{minipage} 
\centering
\begin{minipage}{1.9 cm} 
\centering 
% % 
\begin{tabular}{|c|c|} 
\hline
$b$ & $L$ \\  \hline
\hline
$1690$ & $2$ \\ 
 \hline 
$1716$ & $8$ \\ 
 \hline 
$1728$ & $4$ \\ 
 \hline 
$1755$ & $6$ \\ 
 \hline 
$1792$ & $4$ \\ 
 \hline 
$1848$ & $5$ \\ 
 \hline 
$1872$ & $4$ \\ 
 \hline 
$1875$ & $5$ \\ 
 \hline 
$1890$ & $3$ \\ 
 \hline 
$1960$ & $4$ \\ 
 \hline 
$1980$ & $5$ \\ 
 \hline 
$2016$ & $3$ \\ 
 \hline 
$2028$ & $4$ \\ 
 \hline 
$2080$ & $3$ \\ 
 \hline 
$2100$ & $3$ \\ 
 \hline 
$2106$ & $3$ \\ 
 \hline 
$2112$ & $4$ \\ 
 \hline 
$2156$ & $5$ \\ 
 \hline 
$2160$ & $5$ \\ 
 \hline 
$2178$ & $2$ \\ 
 \hline 
$2184$ & $6$ \\ 
 \hline 
$2200$ & $2$ \\ 
 \hline 
$2250$ & $6$ \\ 
 \hline 
$2268$ & $9$ \\ 
 \hline 
$2275$ & $4$ \\ 
 \hline 
$2310$ & $2$ \\ 
 \hline 
$2340$ & $3$ \\ 
 \hline 
$2366$ & $2$ \\ 
 \hline 
$2464$ & $4$ \\ 
 \hline 
$2496$ & $4$ \\ 
 \hline 
$2520$ & $4$ \\ 
 \hline 
$2548$ & $5$ \\ 
 \hline 
$2574$ & $4$ \\ 
 \hline 
$2600$ & $4$ \\ 
 \hline 
$2640$ & $2$ \\ 
 \hline 
$2704$ & $4$ \\ 
 \hline 
$2730$ & $3$ \\ 
 \hline 
$2750$ & $2$ \\ 
 \hline 
$2772$ & $8$ \\ 
 \hline 
$2808$ & $4$ \\ 
 \hline 
$2816$ & $4$ \\ 
 \hline 
$2912$ & $4$ \\ 
 \hline 
$2925$ & $5$ \\ 
 \hline 
$3024$ & $2$ \\ 
 \hline 
$3080$ & $3$ \\ 
 \hline 
$3120$ & $3$ \\ 
 \hline
 $3136$ & $4$ \\ 
 \hline 
 $3234$ & $2$ \\ 
 \hline 
 $3276$ & $3$ \\ 
 \hline 
$3300$ & $3$ \\ 
 \hline 
 $3360$ & $5$ \\ 
 \hline 
\end{tabular} 
\end{minipage} 
\centering
\begin{minipage}{1.9 cm} 
\centering 
% % 
\begin{tabular}{|c|c|} 
\hline
$b$ & $L$ \\  \hline
\hline 
$3432$ & $4$ \\ 
 \hline 
$3510$ & $9$ \\ 
 \hline 
$3640$ & $4$ \\ 
 \hline 
$3696$ & $3$ \\ 
 \hline 
$3744$ & $3$ \\ 
 \hline 
$3780$ & $2$ \\ 
 \hline 
$3822$ & $3$ \\ 
 \hline 
$3960$ & $2$ \\ 
 \hline 
$4095$ & $4$ \\ 
 \hline 
$4160$ & $4$ \\ 
 \hline 
$4224$ & $2$ \\ 
 \hline 
$4290$ & $4$ \\ 
 \hline 
$4312$ & $2$ \\ 
 \hline 
$4320$ & $2$ \\ 
 \hline 
$4368$ & $3$ \\ 
 \hline 
$4550$ & $2$ \\ 
 \hline 
$4732$ & $4$ \\ 
 \hline 
$5040$ & $3$ \\ 
 \hline 
$5096$ & $2$ \\ 
 \hline 
$5148$ & $3$ \\ 
 \hline 
$5200$ & $2$ \\ 
 \hline 
$5280$ & $4$ \\ 
 \hline 
$5544$ & $2$ \\ 
 \hline 
$5632$ & $3$ \\ 
 \hline 
$5824$ & $4$ \\ 
 \hline 
$5850$ & $4$ \\ 
 \hline 
$6006$ & $4$ \\ 
 \hline 
$6160$ & $2$ \\ 
 \hline 
$6240$ & $4$ \\ 
 \hline 
$6600$ & $4$ \\ 
 \hline 
$6825$ & $3$ \\ 
 \hline 
$7040$ & $4$ \\ 
 \hline 
$7392$ & $2$ \\ 
 \hline 
$7800$ & $3$ \\ 
 \hline 
$8190$ & $2$ \\ 
 \hline 
$8316$ & $3$ \\ 
 \hline 
$8580$ & $2$ \\ 
 \hline 
$8775$ & $3$ \\ 
 \hline 
$9100$ & $4$ \\ 
 \hline 
$9240$ & $4$ \\ 
 \hline 
$10080$ & $2$ \\ 
 \hline 
$10296$ & $2$ \\ 
 \hline 
$11264$ & $2$ \\ 
 \hline 
$11700$ & $2$ \\ 
 \hline 
$12012$ & $4$ \\ 
 \hline 
 $13650$ & $3$ \\ 
 \hline 
$18018$ & $3$ \\ 
 \hline
\end{tabular} 
\end{minipage} 
\centering
\end{table}
%%%%% Both Covs unused primes but used moduli

\begin{table}[tp]
\footnotesize
\centering
\caption{Number of primes, $M=M(b)$, not used in both coverings}\label{orderofunusedprimes1}
\begin{minipage}{1.9 cm} 
\centering 
% % 
\begin{tabular}{|c|c|} 
\hline
$b$ & $M$ \\  \hline
\hline 
$225$ & $1$ \\ 
 \hline 
$288$ & $1$ \\ 
 \hline 
$300$ & $1$ \\ 
 \hline 
$350$ & $2$ \\ 
 \hline 
$637$ & $1$ \\ 
 \hline 
$960$ & $1$ \\ 
 \hline 
$968$ & $1^*$ \\ 
 \hline 
$1080$ & $1$ \\ 
 \hline 
$1120$ & $1$ \\ 
 \hline 
$1125$ & $1^*$ \\ 
 \hline 
$1155$ & $1$ \\ 
 \hline 
$1232$ & $1$ \\ 
 \hline 
$1250$ & $2$ \\ 
 \hline 
$1260$ & $1$ \\ 
 \hline 
$1350$ & $2^*$ \\ 
 \hline 
$1352$ & $1$ \\ 
 \hline 
$1430$ & $3^*$ \\ 
 \hline 
$1452$ & $3^*$ \\ 
 \hline 
$1485$ & $1^*$ \\ 
 \hline 
$1568$ & $1$ \\ 
 \hline 
$1584$ & $2$ \\ 
 \hline 
$1620$ & $6$ \\ 
 \hline 
$1625$ & $2^*$ \\ 
 \hline 
$1664$ & $1^*$ \\ 
 \hline 
$1690$ & $2^*$ \\ 
 \hline 
$1792$ & $1^*$ \\ 
 \hline 
$1875$ & $1^*$ \\ 
 \hline 
$1960$ & $3^*$ \\ 
 \hline 
$2016$ & $1^*$ \\ 
 \hline 
$2028$ & $2^*$ \\ 
 \hline 
 $2112$ & $1^*$ \\ 
 \hline 
 \end{tabular} 
\end{minipage} 
\centering
\begin{minipage}{1.9 cm} 
\centering 
% % 
\begin{tabular}{|c|c|} 
\hline
$b$ & $M$ \\  \hline
\hline
$2200$ & $3^*$ \\ 
 \hline 
$2310$ & $2^*$ \\ 
 \hline 
$2340$ & $4^*$ \\ 
 \hline 
$2704$ & $1^*$ \\ 
 \hline 
$2730$ & $1^*$ \\ 
 \hline 
$2750$ & $4^*$ \\ 
 \hline 
$3024$ & $1^*$ \\ 
 \hline 
$3120$ & $1^*$ \\ 
 \hline 
$3136$ & $1^*$ \\ 
 \hline 
$3234$ & $1^*$ \\ 
 \hline 
$3276$ & $1^*$ \\ 
 \hline 
$3744$ & $1^*$ \\ 
 \hline 
$3780$ & $2^*$ \\ 
 \hline 
$3960$ & $2^*$ \\ 
 \hline 
$4312$ & $1$ \\ 
 \hline 
$4320$ & $4^*$ \\ 
 \hline 
$4368$ & $1^*$ \\ 
 \hline 
$4732$ & $2^*$ \\ 
 \hline 
$5096$ & $1^*$ \\ 
 \hline 
$5200$ & $3^*$ \\ 
 \hline 
$6240$ & $2^*$ \\ 
 \hline 
$6825$ & $1^*$ \\ 
 \hline 
$7392$ & $2^*$ \\ 
 \hline 
$7800$ & $3^*$ \\ 
 \hline 
$8316$ & $1^*$ \\ 
 \hline 
$8580$ & $3^*$ \\ 
 \hline 
$10296$ & $1^*$ \\ 
 \hline 
$11264$ & $1^*$ \\ 
 \hline 
$11700$ & $4^*$ \\ 
 \hline 
$12012$ & $1^*$ \\ 
 \hline

 \hline 

\end{tabular} 
\end{minipage} 
\centering
\end{table}

%%%%% COVERINGS FOR S AND R %%%%%

\begin{table}[h!] 
\footnotesize 
\centering 
\caption{Covering information for Sierpi\'{n}ski numbers}
\label{covSier}
\begin{minipage}{4.1cm} 
 \centering 
 \begin{tabular}{|c|c|} 
 \hline 
 congruence & $p$ \\ \hline 
 $n\equiv 0 \pmod{4}$ & 1 \\ \hline 
$n\equiv 2 \pmod{16}$ & 1 \\ \hline 
$n\equiv 6 \pmod{32}$ & 1 \\ \hline 
$n\equiv 22 \pmod{64}$ & 1 \\ \hline 
$n\equiv 54 \pmod{64}$ & 2 \\ \hline 
$n\equiv 10 \pmod{48}$ & 1 \\ \hline 
$n\equiv 26 \pmod{96}$ & 1 \\
\hline 
 \end{tabular} 
 \end{minipage} 
 \centering 
\begin{minipage}{4.1cm} 
 \centering 
 \begin{tabular}{|c|c|} 
 \hline 
 congruence & $p$ \\ \hline 
$n\equiv 782 \pmod{3136}$ & 1 \\ \hline 
$n\equiv 1566 \pmod{3136}$ & 2 \\ \hline 
$n\equiv 2350 \pmod{3136}$ & 3 \\ \hline 
$n\equiv 3134 \pmod{3136}$ & 4 \\ \hline 
$n\equiv 0 \pmod{13}$ & 1 \\ \hline 
$n\equiv 1 \pmod{26}$ & 1 \\ \hline 
$n\equiv 2 \pmod{39}$ & 1 \\
\hline 
 \end{tabular} 
 \end{minipage} 
 \centering 
\begin{minipage}{4.1cm} 
 \centering 
 \begin{tabular}{|c|c|} 
 \hline 
 congruence & $p$ \\ \hline 
$n\equiv 199 \pmod{520}$ & 2 \\ \hline 
$n\equiv 303 \pmod{520}$ & 3 \\ \hline 
$n\equiv 407 \pmod{520}$ & 4 \\ \hline 
$n\equiv 511 \pmod{520}$ & 5 \\ \hline 
$n\equiv 5 \pmod{208}$ & 1 \\ \hline 
$n\equiv 31 \pmod{416}$ & 1 \\ \hline 
$n\equiv 239 \pmod{416}$ & 2 \\
\hline 
 \end{tabular} 
 \end{minipage} 
 \centering 
 \end{table}
 \addtocounter{table}{-1}
 \begin{table}[h!] 
\footnotesize 
\centering 
\caption{Covering information for Sierpi\'{n}ski numbers cont.}
%\label{covSier}
\begin{minipage}{4.1cm} 
 \centering 
 \begin{tabular}{|c|c|} 
 \hline 
 congruence & $p$ \\ \hline 
$n\equiv 669 \pmod{1404}$ & 3 \\ \hline 
$n\equiv 903 \pmod{1404}$ & 4 \\ \hline 
$n\equiv 1137 \pmod{2808}$ & 1 \\ \hline 
$n\equiv 2541 \pmod{2808}$ & 2 \\ \hline 
$n\equiv 1371 \pmod{2808}$ & 3 \\ \hline 
$n\equiv 2775 \pmod{2808}$ & 4 \\ \hline 
$n\equiv 71 \pmod{858}$ & 1 \\
\hline 
 \end{tabular} 
 \end{minipage} 
 \centering 
\begin{minipage}{4.1cm} 
 \centering 
 \begin{tabular}{|c|c|} 
 \hline 
 congruence & $p$ \\ \hline 
$n\equiv 449 \pmod{650}$ & 4 \\ \hline 
$n\equiv 579 \pmod{650}$ & 5 \\ \hline 
$n\equiv 85 \pmod{1170}$ & 1 \\ \hline 
$n\equiv 215 \pmod{1170}$ & 2 \\ \hline 
$n\equiv 345 \pmod{1170}$ & 3 \\ \hline 
$n\equiv 475 \pmod{1170}$ & 4 \\ \hline 
$n\equiv 605 \pmod{1170}$ & 5 \\
\hline 
 \end{tabular} 
 \end{minipage} 
 \centering 
\begin{minipage}{4.1cm} 
 \centering 
 \begin{tabular}{|c|c|} 
 \hline 
 congruence & $p$ \\ \hline 
$n\equiv 99 \pmod{325}$ & 2 \\ \hline 
$n\equiv 164 \pmod{975}$ & 1 \\ \hline 
$n\equiv 489 \pmod{975}$ & 2 \\ \hline 
$n\equiv 814 \pmod{975}$ & 3 \\ \hline 
$n\equiv 229 \pmod{975}$ & 4 \\ \hline 
$n\equiv 554 \pmod{2925}$ & 1 \\ \hline 
$n\equiv 1529 \pmod{2925}$ & 2 \\
\hline 
 \end{tabular} 
 \end{minipage} 
 \centering 
 \end{table}
 \addtocounter{table}{-1}
 \begin{table}[h!] 
\footnotesize 
\centering 
\caption{Covering information for Sierpi\'{n}ski numbers cont.}
%\label{covSier}
\begin{minipage}{4.1cm} 
 \centering 
 \begin{tabular}{|c|c|} 
 \hline 
 congruence & $p$ \\ \hline 
$n\equiv 1815 \pmod{4550}$ & 2 \\ \hline 
$n\equiv 2725 \pmod{9100}$ & 1 \\ \hline 
$n\equiv 7275 \pmod{9100}$ & 2 \\ \hline 
$n\equiv 3635 \pmod{9100}$ & 3 \\ \hline 
$n\equiv 8185 \pmod{9100}$ & 4 \\ \hline 
$n\equiv 4545 \pmod{13650}$ & 1 \\ \hline 
$n\equiv 9095 \pmod{13650}$ & 2 \\
\hline 
 \end{tabular} 
 \end{minipage} 
 \centering 
\begin{minipage}{4.1cm} 
 \centering 
 \begin{tabular}{|c|c|} 
 \hline 
 congruence & $p$ \\ \hline 
$n\equiv 1765 \pmod{2730}$ & 2 \\ \hline 
$n\equiv 2675 \pmod{2730}$ & 3 \\ \hline 
$n\equiv 153 \pmod{1274}$ & 1 \\ \hline 
$n\equiv 335 \pmod{1274}$ & 2 \\ \hline 
$n\equiv 517 \pmod{1274}$ & 3 \\ \hline 
$n\equiv 699 \pmod{2548}$ & 1 \\ \hline 
$n\equiv 1973 \pmod{2548}$ & 2 \\
\hline 
 \end{tabular} 
 \end{minipage} 
 \centering 
\begin{minipage}{4.1cm} 
 \centering 
 \begin{tabular}{|c|c|} 
 \hline 
 congruence & $p$ \\ \hline 
$n\equiv 25 \pmod{169}$ & 2 \\ \hline 
$n\equiv 38 \pmod{169}$ & 3 \\ \hline 
$n\equiv 51 \pmod{338}$ & 1 \\ \hline 
$n\equiv 233 \pmod{338}$ & 2 \\ \hline 
$n\equiv 77 \pmod{507}$ & 1 \\ \hline 
$n\equiv 753 \pmod{2028}$ & 1 \\ \hline 
$n\equiv 1767 \pmod{2028}$ & 2 \\ \hline 
$n\equiv 415 \pmod{2028}$ & 3 \\ \hline 
$n\equiv 1429 \pmod{2028}$ & 4 \\ \hline 
$n\equiv 259 \pmod{676}$ & 1 \\ \hline 
$n\equiv 597 \pmod{676}$ & 2 \\ \hline 
$n\equiv 103 \pmod{676}$ & 3 \\ \hline 
$n\equiv 441 \pmod{676}$ & 4 \\ \hline 
$n\equiv 285 \pmod{676}$ & 5 \\ \hline 
$n\equiv 623 \pmod{2704}$ & 1 \\ \hline 
$n\equiv 1299 \pmod{2704}$ & 2 \\ \hline 
$n\equiv 1975 \pmod{2704}$ & 3 \\ \hline 
$n\equiv 2651 \pmod{2704}$ & 4 \\ \hline 
$n\equiv 129 \pmod{845}$ & 1 \\ \hline 
$n\equiv 298 \pmod{845}$ & 2 \\ \hline 
$n\equiv 467 \pmod{845}$ & 3 \\ \hline 
$n\equiv 1481 \pmod{1690}$ & 1 \\ \hline 
$n\equiv 805 \pmod{1690}$ & 2 \\ \hline 
$n\equiv 311 \pmod{1014}$ & 1 \\ \hline 
$n\equiv 649 \pmod{1014}$ & 2 \\ \hline 
$n\equiv 987 \pmod{1014}$ & 3 \\ \hline 
$n\equiv 155 \pmod{1183}$ & 1 \\ \hline 
$n\equiv 324 \pmod{1183}$ & 2 \\ \hline 
$n\equiv 493 \pmod{1183}$ & 3 \\ \hline 
$n\equiv 1845 \pmod{2366}$ & 1 \\ \hline 
$n\equiv 831 \pmod{2366}$ & 2 \\ \hline 
$n\equiv 2183 \pmod{4732}$ & 1 \\ \hline 
$n\equiv 4549 \pmod{4732}$ & 2 \\ \hline 
$n\equiv 1169 \pmod{4732}$ & 3 \\ \hline 
$n\equiv 3535 \pmod{4732}$ & 4 \\ \hline 
$n\equiv 337 \pmod{1352}$ & 1 \\ \hline 
$n\equiv 675 \pmod{1352}$ & 2 \\ \hline 
$n\equiv 1013 \pmod{1352}$ & 3 \\ \hline 
$n\equiv 1351 \pmod{1352}$ & 4 \\ \hline 
 \end{tabular} 
 \end{minipage} 
 \centering 
\end{table}

\begin{table}[h!] 
\footnotesize 
\centering 
\caption{Covering information for Riesel numbers}
\label{covRies}
\begin{minipage}{4.1cm} 
 \centering 
 \begin{tabular}{|c|c|} 
 \hline 
 congruence & $p$ \\ \hline 
 $n\equiv 2 \pmod{4}$ & 1 \\ \hline 
$n\equiv 4 \pmod{40}$ & 1 \\ \hline 
$n\equiv 24 \pmod{80}$ & 1 \\ \hline 
$n\equiv 64 \pmod{160}$ & 1 \\ \hline 
$n\equiv 144 \pmod{160}$ & 2 \\ \hline 
$n\equiv 8 \pmod{120}$ & 1 \\ \hline 
$n\equiv 28 \pmod{240}$ & 1 \\
\hline 
 \end{tabular} 
 \end{minipage} 
 \centering 
\begin{minipage}{4.1cm} 
 \centering 
 \begin{tabular}{|c|c|} 
 \hline 
 congruence & $p$ \\ \hline 
$n\equiv 72 \pmod{200}$ & 3 \\ \hline 
$n\equiv 172 \pmod{200}$ & 4 \\ \hline 
$n\equiv 92 \pmod{300}$ & 1 \\ \hline 
$n\equiv 192 \pmod{300}$ & 2 \\ \hline 
$n\equiv 292 \pmod{300}$ & 3 \\ \hline 
$n\equiv 16 \pmod{140}$ & 1 \\ \hline 
$n\equiv 36 \pmod{140}$ & 2 \\
\hline 
 \end{tabular} 
 \end{minipage} 
 \centering 
\begin{minipage}{4.1cm} 
 \centering 
 \begin{tabular}{|c|c|} 
 \hline 
 congruence & $p$ \\ \hline 
$n\equiv 2085 \pmod{2250}$ & 3 \\ \hline 
$n\equiv 735 \pmod{2250}$ & 4 \\ \hline 
$n\equiv 1485 \pmod{2250}$ & 5 \\ \hline 
$n\equiv 2235 \pmod{2250}$ & 6 \\ \hline 
$n\equiv 15 \pmod{75}$ & 1 \\ \hline 
$n\equiv 40 \pmod{75}$ & 2 \\ \hline 
$n\equiv 65 \pmod{225}$ & 1 \\
\hline 
 \end{tabular} 
 \end{minipage} 
 \centering 
 \end{table}
 \addtocounter{table}{-1}
 \begin{table}[h!] 
\footnotesize 
\centering 
\caption{Covering information for Riesel numbers cont.}
%\label{covRies}
\begin{minipage}{4.1cm} 
 \centering 
 \begin{tabular}{|c|c|} 
 \hline 
 congruence & $p$ \\ \hline 
$n\equiv 1275 \pmod{2268}$ & 8 \\ \hline 
$n\equiv 2031 \pmod{2268}$ & 9 \\ \hline 
$n\equiv 627 \pmod{1512}$ & 1 \\ \hline 
$n\equiv 1383 \pmod{3024}$ & 1 \\ \hline 
$n\equiv 2895 \pmod{3024}$ & 2 \\ \hline 
$n\equiv 9 \pmod{135}$ & 1 \\ \hline 
$n\equiv 36 \pmod{135}$ & 2 \\
\hline
 \end{tabular} 
 \end{minipage} 
 \centering 
\begin{minipage}{4.1cm} 
 \centering 
 \begin{tabular}{|c|c|} 
 \hline 
 congruence & $p$ \\ \hline  
$n\equiv 753 \pmod{1890}$ & 1 \\ \hline 
$n\equiv 1131 \pmod{1890}$ & 2 \\ \hline 
$n\equiv 1509 \pmod{1890}$ & 3 \\ \hline 
$n\equiv 1887 \pmod{3780}$ & 1 \\ \hline 
$n\equiv 3777 \pmod{3780}$ & 2 \\ \hline 
$n\equiv 11 \pmod{33}$ & 1 \\ \hline 
$n\equiv 22 \pmod{99}$ & 1 \\
\hline 
 \end{tabular} 
 \end{minipage} 
 \centering 
\begin{minipage}{4.1cm} 
 \centering 
 \begin{tabular}{|c|c|} 
 \hline 
 congruence & $p$ \\ \hline 
$n\equiv 817 \pmod{1540}$ & 3 \\ \hline 
$n\equiv 1433 \pmod{1540}$ & 4 \\ \hline 
$n\equiv 58 \pmod{385}$ & 1 \\ \hline 
$n\equiv 212 \pmod{385}$ & 2 \\ \hline 
$n\equiv 289 \pmod{385}$ & 3 \\ \hline 
$n\equiv 751 \pmod{1155}$ & 1 \\ \hline 
$n\equiv 1136 \pmod{1155}$ & 2 \\
\hline
 \end{tabular} 
 \end{minipage} 
 \centering 
 \end{table}
 \addtocounter{table}{-1}
 \begin{table}[h!] 
\footnotesize 
\centering 
\caption{Covering information for Riesel numbers cont.}
%\label{covRies}
\begin{minipage}{4.1cm} 
 \centering 
 \begin{tabular}{|c|c|} 
 \hline 
 congruence & $p$ \\ \hline  
$n\equiv 863 \pmod{1100}$ & 4 \\ \hline 
$n\equiv 423 \pmod{1100}$ & 5 \\ \hline 
$n\equiv 973 \pmod{1100}$ & 6 \\ \hline 
$n\equiv 533 \pmod{1100}$ & 7 \\ \hline 
$n\equiv 1083 \pmod{2200}$ & 1 \\ \hline 
$n\equiv 2183 \pmod{2200}$ & 2 \\ \hline 
$n\equiv 6 \pmod{121}$ & 1 \\
\hline 
 \end{tabular} 
 \end{minipage} 
 \centering 
\begin{minipage}{4.1cm} 
 \centering 
 \begin{tabular}{|c|c|} 
 \hline 
 congruence & $p$ \\ \hline 
$n\equiv 6563 \pmod{6600}$ & 4 \\ \hline 
$n\equiv 19 \pmod{154}$ & 1 \\ \hline 
$n\equiv 41 \pmod{154}$ & 2 \\ \hline 
$n\equiv 239 \pmod{770}$ & 1 \\ \hline 
$n\equiv 393 \pmod{770}$ & 2 \\ \hline 
$n\equiv 547 \pmod{770}$ & 3 \\ \hline 
$n\equiv 701 \pmod{2310}$ & 1 \\
\hline 
 \end{tabular} 
 \end{minipage} 
 \centering 
\begin{minipage}{4.1cm} 
 \centering 
 \begin{tabular}{|c|c|} 
 \hline 
 congruence & $p$ \\ \hline 
$n\equiv 449 \pmod{704}$ & 3 \\ \hline 
$n\equiv 625 \pmod{704}$ & 4 \\ \hline 
$n\equiv 119 \pmod{880}$ & 1 \\ \hline 
$n\equiv 471 \pmod{880}$ & 2 \\ \hline 
$n\equiv 647 \pmod{880}$ & 3 \\ \hline 
$n\equiv 823 \pmod{880}$ & 4 \\ \hline 
$n\equiv 317 \pmod{1056}$ & 1 \\ \hline 
$n\equiv 493 \pmod{1056}$ & 2 \\ \hline 
$n\equiv 845 \pmod{1056}$ & 3 \\ \hline 
$n\equiv 1021 \pmod{1056}$ & 4 \\ \hline 
$n\equiv 163 \pmod{1408}$ & 1 \\ \hline 
$n\equiv 339 \pmod{1408}$ & 2 \\ \hline 
$n\equiv 515 \pmod{1408}$ & 3 \\ \hline 
$n\equiv 691 \pmod{2816}$ & 1 \\ \hline 
$n\equiv 2099 \pmod{2816}$ & 2 \\ \hline 
$n\equiv 867 \pmod{2816}$ & 3 \\ \hline 
$n\equiv 2275 \pmod{2816}$ & 4 \\ \hline 
$n\equiv 1043 \pmod{4224}$ & 1 \\ \hline 
$n\equiv 3859 \pmod{4224}$ & 2 \\ \hline 
$n\equiv 1219 \pmod{5632}$ & 1 \\ \hline 
$n\equiv 2627 \pmod{5632}$ & 2 \\ \hline 
$n\equiv 4035 \pmod{5632}$ & 3 \\ \hline 
$n\equiv 5443 \pmod{11264}$ & 1 \\ \hline 
$n\equiv 11075 \pmod{11264}$ & 2 \\ \hline 
$n\equiv 2803 \pmod{7040}$ & 1 \\ \hline 
$n\equiv 4211 \pmod{7040}$ & 2 \\ \hline 
$n\equiv 5619 \pmod{7040}$ & 3 \\ \hline 
$n\equiv 7027 \pmod{7040}$ & 4 \\ \hline 
$n\equiv 43 \pmod{198}$ & 1 \\ \hline 
$n\equiv 65 \pmod{198}$ & 2 \\ \hline 
$n\equiv 109 \pmod{396}$ & 1 \\ \hline 
$n\equiv 307 \pmod{396}$ & 2 \\ \hline 
$n\equiv 131 \pmod{396}$ & 3 \\ \hline 
$n\equiv 329 \pmod{396}$ & 4 \\ \hline 
$n\equiv 175 \pmod{594}$ & 1 \\ \hline 
$n\equiv 373 \pmod{594}$ & 2 \\ \hline 
$n\equiv 571 \pmod{594}$ & 3 \\ \hline 
$n\equiv 197 \pmod{1188}$ & 1 \\ \hline 
$n\equiv 395 \pmod{1188}$ & 2 \\ \hline 
$n\equiv 593 \pmod{1188}$ & 3 \\ \hline 
$n\equiv 791 \pmod{1188}$ & 4 \\ \hline 
$n\equiv 989 \pmod{1188}$ & 5 \\ \hline 
$n\equiv 1187 \pmod{1188}$ & 6 \\ \hline 
 \end{tabular} 
 \end{minipage} 
 \centering 
 \end{table}
\end{document}